\documentclass[12pt]{article}
\usepackage[utf8]{inputenc}
\usepackage{amsmath}
\usepackage[english]{babel}
\usepackage{url}

\usepackage[a4paper,top=3cm,bottom=2cm,left=3cm,right=3cm,marginparwidth=1.75cm]{geometry}

\usepackage{amsmath}
\usepackage{graphicx}
\usepackage[colorinlistoftodos]{todonotes}
\usepackage[colorlinks=true, allcolors=blue]{hyperref}
\usepackage{amssymb}
\usepackage{amssymb}
\usepackage{amsthm}
\usepackage{authblk}
\usepackage{amsmath}
\usepackage{amsfonts}
\usepackage{amssymb}
\usepackage{amscd}
\usepackage{latexsym}
\usepackage{breqn}
\usepackage{cite}
\usepackage{mathtools}
\mathtoolsset{showonlyrefs}
\usepackage{fancyhdr}

\theoremstyle{plain}

\theoremstyle{definition}

\theoremstyle{remark}

\usepackage{subcaption}
\usepackage{cite}
\usepackage{comment}
\usepackage{mathtools}
\mathtoolsset{showonlyrefs}

\makeatletter
\newenvironment{leftalign*}[1][\parindent]{\setlength\hangindent{#1}\start@align\tw@\st@rredtrue\m@ne}{\endalign}

\makeatother

\usepackage{tikz,amsmath,amsfonts}
\usetikzlibrary{calc,intersections}
\usetikzlibrary{decorations.pathreplacing}
\usepackage{ifthen}
\renewcommand{\phi}{\varphi}
\renewcommand{\Re}[1]{\operatorname{Re} #1 }

\newcommand{\ran}{\operatorname{ran}}

\newcommand{\tr}{\operatorname{tr}}
\let\oldenumerate=\enumerate
	\def\enumerate{
	\oldenumerate
	\setlength{\itemsep}{5pt}
	}
\let\olditemize=\itemize
	\def\itemize{
	\olditemize
	\setlength{\itemsep}{5pt}
	}

\newtheorem{Theorem}{Theorem}[section]
\newtheorem{Lemma}[Theorem]{Lemma}

\newtheorem{Corollary}[Theorem]{Corollary}

\newtheorem{Question}[Theorem]{Question}
\newtheorem{Example}[Theorem]{Example}
\newtheorem{Remark}[Theorem]{Remark}

\allowdisplaybreaks

\numberwithin{equation}{section}
\title{Describing the Numerical Range and $C$-Numerical Range of Matrices via Their Unitary Orbits and the Joukowsky Transform}
\author{Ryan O'Loughlin\\

Department of Mathematics and Statistics, \\
University of Reading \\
Reading \\
RG6 6AX \\
}
\date{}

\begin{document}

\maketitle

\begin{abstract}
We generalise the Elliptical Range Theorem to characterise the numerical range of matrices belonging to a subspace of the space of \(3 \times 3\) matrices. Using Specht's Theorem, which characterizes when two matrices are unitarily equivalent, we then provide a novel proof of the Elliptical Range Theorem. Finally, we give an explicit description of the $C$-numerical range for $2 \times 2$ matrices and for rank-one matrices.
\end{abstract}

\begin{section}{Introduction}
The numerical range of a $n \times n$ matrix $A$ (also known as the field of values) is defined as 
    $$
    W(A) = \{ \langle Ax,x\rangle : x \in \mathbb{C}^n, \,  \|x\| = 1 \},
    $$
where $\langle\, , \, \rangle $ denotes the standard Euclidean inner product on $\mathbb{C}^n$.

The Elliptical Range Theorem describes the numerical range of a $ 2 \times 2$ matrix as an ellipse. In this article we adapt the typical proof of the Elliptical Range Theorem to incorporate the Joukowsky transform, which consequently generalises to give a new description for the numerical range of a class of $3 \times 3$ matrices. Then through unitary equivalence conditions (including Specht's Theorem) we give another alternate new proof of the Elliptical Range Theorem and describe the $C$-numerical range of $2 \times 2$ matrices and rank $1$ matrices.

Although numerical ranges have a long history there is no complete precise mathematical description describing their geometry. The Toeplitz-Hausdorff Theorem shows that numerical ranges are convex and there are now many proofs of the Elliptical Range Theorem, with \cite[Thm 1.3.6]{johnson1985matrix} being a classical proof and \cite{li1996simple, lakos2024short} being shorter alternatives. Although a modern standard method to analyse the numerical range is based on Kippenhahn's work \cite{kippenhahn1951wertevorrat} by computing the Kippenhanh polynomials, despite many years of progress since Kippenhanh's work this is still an involved task, and the polynomials in question often involve thousands of terms. We hope that the alternative method involving Laurent polynomials (Theorem \ref{3by3J}) presented in this article will contribute to a more practical characterisation of the numerical range, with the recovery of the Elliptical Range Theorem in the $ 2 \times 2$ case serving as a promising first step in this direction.

In a similar vein, the $C$-numerical range of a matrix (defined in Section \ref{3n}) has no complete geometric description. The $C$-numerical range of a $2 \times 2$ matrix was shown to be an ellipse in  \cite{nakazato1994c}, and a subsequent proof was given in \cite{kwong1996some}. Our new approach incorporating the Joukowsky transform shortens these proofs and in our proof we are able to explicitly identify the ellipse which is the $C$-numerical range of a given matrix. The influential "union of disks" formula was first shown for $q$-numerical ranges in \cite{tsing1984constrained}, and has led to substantial progress on study of the $q$-numerical ranges. With a novel proof, we make a multidimensional generalisation of this formula to rank 1 $C$-numerical ranges.

One particular reason why one may want to describe the geometry of the numerical range is in order to help solve Crouzeix's Conjecture. This conjectures that for each polynomial $p$ and matrix $A$, we have $\|p(A)\| \leq 2 \sup_{z \in W(A)}|p(z)|$, where $\| \, \|$ denotes the operator norm. Crouzeix's Conjecture has gained a lot of traction from the Functional and Complex Analysis communities \cite{o2024crouzeix, CC11.08, CC2matrix, CC3by3, CCgorkin, CCincreasing, CCLi, CCLi2, CCnilpotent, CCnum, malman2024double}. In particular Crouzeix`s Conjecture is unresolved for $ 3 \times 3$ matrices which have elliptical numerical ranges, which is the focus of Theorem \ref{3by3J}.

The layout of the paper is as follows. In Section \ref{2.1n} we provide a short proof of the Elliptical Range Theorem, and generalise this to $3 \times 3$ matrices in Section \ref{2.2}. In Section \ref{2.3} we provide a novel proof the Elliptical Range Theorem through the use of Specht's Theorem. Section \ref{3n} is focused on the $C$-numerical range, with Section \ref{3.1} describing the $C$-numerical range of a $2 \times 2$ matrix and Section \ref{3.2} describing the $C$-numerical range of a rank $1$ matrix.

\begin{subsection}{Notation, Terminology and Preliminaries}\label{1.1}

We use $\operatorname{tr}(A)$ to denote the trace of a matrix $A$. We use the Hilbert-Schmidt norm notation 
$$
\left\| \begin{pmatrix}
    a_{11} & a_{12} \\
    a_{21} & a_{22} 
\end{pmatrix} \right\|_{HS}^2 = |a_{11}|^2 + |a_{12}|^2 + |a_{21}|^2 + |a_{22}|^2 = \operatorname{tr}\left(\begin{pmatrix}
    a_{11} & a_{12} \\
    a_{21} & a_{22} 
\end{pmatrix}^*\begin{pmatrix}
    a_{11} & a_{12} \\
    a_{21} & a_{22} 
\end{pmatrix}\right).
$$
For a set $X\subseteq \mathbb{C}$, we write $\operatorname{conv}(X)$ to denote the convex hull of $X$.

Throughout this article, the term ellipse will refer to the convex hull of the ellipse. When we wish to consider only the boundary, this will be stated explicitly. We use the notation $\partial W(A)$ to denote the boundary of the numerical range of $A$.

For parameters $a, b \in \mathbb{R}$, we define the Joukowsky transform on the unit circle, $\mathbb{T}$, by
\[
J_{a,b}(z) = az + \frac{b}{z}.
\]
Writing $
e^{it} =  \cos (t) + i  \sin (t), \,  e^{-it} =  \cos (t) - i  \sin (t),$ one quickly sees that
\begin{equation}\label{Jtransform}
    J_{a,b}(e^{it}) = (a+b) \cos (t) + i(a-b) \sin (t).
\end{equation}

Observing the range of the Joukowsky transform as the boundary of an ellipse, allows us to present a quick proof of the Elliptical Range Theorem in Theorem \ref{2ndproof}. We then generalise this approach to describe the numerical range of $3 \times 3$ matrices in terms of unions of ranges of the Joukowsky transform (Theorem \ref{3by3J}).

A matrix $U$ is unitary if $U^* = U^{-1}$, where $U^*$ denotes the adjoint of $U$. Matrices $A$ and $B$ are unitarily equivalent if $A = U^*BU$. It is well-known that matrices which are unitarily equivalent have the same numerical range.

Our new second proof of the Elliptical Range Theorem is deduced from the following.

\begin{Theorem}{Specht's Theorem}\label{Specht}
    \begin{enumerate}
        \item Let $A$ and $B$ be $2 \times 2$ matrices. Then $A$ and $B$ are unitarily equivalent if and only if $\operatorname{tr} A=\operatorname{tr} B, \quad \operatorname{tr} A^2=\operatorname{tr} B^2,  \text { and }\operatorname{tr} A A^*=\operatorname{tr} B B^*$.
        \item Let  $A$ and $B$ be $3 \times 3$ matrices. Then $A$ and $B$ are unitarily equivalent if and only if the following trace inequalities hold \begin{align}\label{all}
    &\operatorname{tr} A=\operatorname{tr} B, \quad \operatorname{tr} A^2=\operatorname{tr} B^2, \quad \operatorname{tr} A A^*=\operatorname{tr} B B^*  \quad \operatorname{tr} A^3=\operatorname{tr} B^3 \\
    &\operatorname{tr} A^2 A^*=\operatorname{tr} B^2 B^*,  \quad \operatorname{tr} A^2\left(A^*\right)^2=\operatorname{tr} B^2\left(B^*\right)^2, \quad \\ &\operatorname{tr} A^2\left(A^*\right)^2 A A^*=\operatorname{tr} B^2\left(B^*\right)^2 B B^*.
    \end{align}
    \end{enumerate}
\end{Theorem}

Part $(a)$ of this theorem is \cite[Thm 2.2.8]{johnson1985matrix}, and part $(b)$ can be found in \cite{djokovic2007unitarily}.

\end{subsection}
\end{section}

\begin{section}{The Elliptical Range Theorem and Generalisations}\label{S2}
\subsection{A Short Proof Via the Joukowsky Transform}\label{2.1n}

The proof of the Elliptical Range Theorem below uses the same method as that given in Section 1 of \cite{hornjohnson}, but is shortened by incorporating Joukowsky transforms.

\begin{Theorem}\label{2ndproof}
    Let $A$ be a $ 2 \times 2$ matrix. Let $$
    \theta = \frac{\pi - \operatorname{arg}(\det (A - \frac{\operatorname{tr}(A)}{2}))}{2} 
    $$
    and $A_0 = e^{i \theta}(A - \frac{\operatorname{tr}(A)}{2})$. Then $W(A)$ is an ellipse with boundary parametrised by 
$$
t \mapsto \frac{\operatorname{tr}(A)}{2} +  e^{-i\theta}\left(\sqrt{\frac{\|A_0\|_{HS}^2}{4} - \frac{\det(A_0)}{2}} \cos(t) + i \sqrt{\frac{\|A_0\|_{HS}^2}{4} + \frac{\det(A_0)}{2}} \sin(t)\right).
$$ 
\end{Theorem}

\begin{proof}
    Since $\operatorname{tr}(A_0) = 0$ it is known (see for example Page 19 in \cite{hornjohnson}) that $A_0$ is unitarily equivalent to $\tilde{A_0} = s \begin{pmatrix}
        0 & b \\
        c & 0 
    \end{pmatrix}$ where $|s| = 1$ and $b \geq c \geq 0$, and since $\det A_0 \leq 0$, we have $s = \pm 1$. We initially assume $s=1$. Any unit vector in $\mathbb{C}^2$ is a unimodular constant multiple of $\begin{pmatrix}
        \cos (\alpha) \\
        e^{i \psi} \sin (\alpha)
    \end{pmatrix}$ for $\psi \in [0, 2 \pi]$, and so
\begin{align}
&W(\tilde{A_0})\\
&= \left\{ \left\langle \begin{pmatrix}
        0 & b \\
        c & 0 
    \end{pmatrix}\begin{pmatrix}
        \cos (\alpha) \\
        e^{i \psi} \sin (\alpha)
    \end{pmatrix}, \begin{pmatrix}
        \cos (\alpha) \\
        e^{i \psi} \sin (\alpha)
    \end{pmatrix} \right\rangle  :  \psi \in [0, 2 \pi]   , \alpha \in [0, 2 \pi ]\right\}\\
    &= \{\sin (\alpha) \cos (\alpha) (be^{i \psi} +  c e^{-i \psi}) : \psi \in [0, 2 \pi] , \,   \alpha \in [0, 2 \pi ]  \} \\
    &= \bigcup_{\alpha \in [0, 2 \pi]} \left\{ \frac{\sin (2 \alpha)}{2} \ran J_{b,c} \right\} \label{Qanswer}.
\end{align}
Observe that since $\ran J_{b,c} = \ran J_{-b,-c} $, in the end $W(\tilde{A_0})$  is the same regardless of whether $s = 1$ or $s=-1$. Observe that \eqref{Qanswer} is a union of concentric boundaries of ellipses, thus is given by  $\operatorname{conv}(\frac{1}{2} \ran J_{b,c})$, which is the ellipse with boundary parametrised by $t \mapsto \frac{b+c}{2}  \cos (t) + i \frac{b-c}{2}  \sin (t)$. Noting that \begin{equation}\label{useforC}
    \frac{b+c}{2} = \sqrt{\frac{\|\tilde{A_0}\|_{HS}^2}{4} - \frac{\det(\tilde{A_0})}{2}}, \, \, \, \frac{b-c}{2} =  \sqrt{\frac{\|\tilde{A_0}\|_{HS}^2}{4} + \frac{\det(\tilde{A_0})}{2}}
\end{equation} and the quantities $\| \, \|_{HS}$ and $\det( \, )$ are invariant under unitary equivalence means $W(A_0)$ is the ellipse with boundary parametrised by
$$
t \mapsto \sqrt{\frac{\|A_0\|_{HS}^2}{4} - \frac{\det(A_0)}{2}} \cos(t) + i \sqrt{\frac{\|A_0\|_{HS}^2}{4} + \frac{\det(A_0)}{2}} \sin(t).
$$
The general case now follows via the linearity relation $W(A_0) = e^{i \theta}W(A) - e^{i \theta}\frac{\tr(A)}{2}$.
\end{proof}
Observe that from \eqref{Qanswer}, even when $b,c$ are complex 
$$
W \left( \begin{pmatrix}
    0 & b \\
     c & 0 
\end{pmatrix} \right) = \frac{\operatorname{conv}(\ran J_{b,c})}{2},
$$
thus we can conclude the following corollary to the theorem above.

\begin{Corollary}
    For $b,c \in \mathbb{C}$, we have $\ran J_{b,c}$ is the boundary of an ellipse.
\end{Corollary}

A natural question arising in numerical range theory is the following, which is often referred to as the inverse numerical range problem. 
\begin{Question}
    Given a point $p \in W(A)$, can one find a unit norm $h$ such that $\langle Ah,h \rangle  = p$?
\end{Question}

Previous results on the inverse numerical range problem appear in \cite{inverseproblem1, inverseproblem2, inverseproblem3}. Inspecting the proof above shows that we can  answer this question in the positive in the case of $2 \times 2$ matrices. Indeed, the $\theta$ parameter corresponds to the boundary of a concentric ellipse $p$ lies on, and the $\psi$ parameter is the value such that $J_{b,c}(e^{i \psi}), p $ and the origin are co-linear. We illustrate this with an example.

Given an arbitrary $2 \times 2 $ matrix $A$ after rotating and translating, we may assume $W(A)$ is centred at the origin and has axis aligned with the real and imaginary axis.

\begin{Example}
Consider $A = \begin{pmatrix}
    0 & 8 \\
    4 & 0
\end{pmatrix}$ and the point $p = 1 + i$. Then by the theorem above $W(A)$ has boundary parametrised by $t \mapsto 6 \cos (t) + 2i \sin (t)$. Since the equation of the boundary of $W(A)$ in Cartesian coordinates is $\frac{x^2}{36} + \frac{y^2}{4} = 1$, to find which boundary of a concentric ellipse the point $1+i$ lies on we solve $\frac{1}{36r^2} + \frac{1}{4r^2} = 1$ for $r$ (this is effectively solving $p \in r \partial(W(A)) ).$ This is solved for $r = \frac{\sqrt{10}}{6}$, so from \eqref{Qanswer}, we see $\theta $ will be such that $\sin ( 2 \theta) = \frac{\sqrt{10}}{6}$. In order to find $\psi$ we solve, $\frac{\sqrt{10}}{12}J_{8,4}(e^{i \psi}) = 1+i$. This amounts to solving the quadratic $2e^{2i \psi} - \frac{3(1+i)}{\sqrt{10}}\,e^{i \psi} + 1 = 0$ in $e^{i \psi}$, which has solution given by $ e^{i \psi} =\frac{1 + 3i}{\sqrt{10}}$. To conclude, we have shown that if 
$$
h = \begin{pmatrix}
        \cos (\theta) \\
        e^{i \psi} \sin (\theta)
    \end{pmatrix},
$$
where $\sin ( 2 \theta) = \frac{\sqrt{10}}{6}$ and $e^{i \psi} =\frac{1 + 3i}{\sqrt{10}}$, then $\langle Ah,h \rangle  = 1 + i$.
\end{Example}

\begin{subsection}{Generalising to $3 \times 3$ matrices}\label{2.2}

It is know that every $3 \times 3$ matrix $A$ is unitarily equivalent to a matrix with $\frac{\operatorname{tr}(A)}{3}$ on the diagonal (see for example Theorem 1.3.4 in \cite{hornjohnson}). Furthermore, we trivially have $W(A  - \frac{\operatorname{tr}(A)}{3}I_d) = W(A) - \frac{\operatorname{tr}(A)}{3}$. Thus, to fully characterize the numerical range of $3 \times 3$ matrices, one may assume the principal diagonal entries are $0$, i.e., $A$ is of the form
\begin{equation}\label{embed}
A =  \begin{pmatrix}
        0 & a_{12} & a_{13} \\
        a_{21} & 0 & a_{23} \\
        a_{31} & a_{32} & 0
    \end{pmatrix}.
\end{equation}
If $a_{13} = a_{31} = a_{23} = a_{32} = 0$, then $A$ is the direct sum of a $2 \times 2$ matrix and zero, so in this case the numerical range is determined by the Elliptical Range Theorem \ref{2ndproof}. Consider instead the family of matrices of the form
\[
A = \begin{pmatrix}
        0 & a_{12} & a_{13} \\
        a_{21} & 0 & 0 \\
        a_{31} & 0 & 0
    \end{pmatrix}.
\]
Within this family, when $a_{13} = a_{31} = 0$ one obtains an isomorphic copy of the space of \(2 \times 2\) matrices of the form  
\[
\begin{pmatrix}
0 & a_{12} \\
a_{21} & 0
\end{pmatrix},
\]
and, in view of Theorem~\ref{2ndproof}, it suffices to restrict to matrices of this type when characterising the numerical range in the \(2 \times 2\) case. At the same time, this family embeds naturally as a subspace of matrices of the form \eqref{embed}. We generalise Theorem \ref{2ndproof} to characterise the numerical range of matrices in this intermediate setting.

In the proof of the next theorem we make use of support functions. For any compact set $K$, the support function in direction $e^{i \phi}$ is defined as
\[
h_K(\phi) = \max_{z \in K} \Re\big( z e^{-i\phi} \big), \, \, \phi \in [0, 2 \pi].
\]
It is known (see for example Appendix 2 in \cite{gau2021numerical}) that if $W(A)$ has no line segment on its boundary then the support function $h_{W(A)}$ is differentiable and the boundary of $W(A)$ is parametrised by
\begin{equation}\label{naturalpara}
    x(\phi) = h_{W(A)}(\phi) \cos\phi - h_{W(A)}'(\phi) \sin\phi, \, \, 
y(\phi) = h_{W(A)}(\phi) \sin\phi + h_{W(A)}'(\phi) \cos\phi,
\end{equation}
where $\phi \in [0, 2 \pi].$ So the support function $h_{W(A)}$ uniquely determines the numerical range. The following lemma is well known, but for completeness we include a proof.

\begin{Lemma}\label{ellipsesupport}
    Let the ellipse $E$ have boundary parametrised by 
    $$
    t \mapsto a \cos (t) + i b \sin (t) , \, \, a,b >0, \, \, t \in [0,2 \pi].
    $$ Then $ h_E(\phi) = \sqrt{a^2\cos^2(\phi) + b^2\sin^2(\phi)}$.
\end{Lemma}
    
\begin{proof}
By definition of support functions $$ h_E(\phi) = \sup_{t \in [0, 2\pi]} (a\cos(t)\cos(\phi) + b\sin(t)\sin(\phi)) .$$
Viewing the above as the inner product of vectors in $\mathbb{R}^2$
$$
\left\langle \begin{pmatrix}
    a \cos(\phi) \\
    b\sin(\phi))
\end{pmatrix},  \begin{pmatrix}
     \cos(t) \\
    \sin(t))
\end{pmatrix} \right\rangle,
$$
and applying Cauchy-Schwarz we obtain the result.
\end{proof}

The main result of the following theorem is part $(b)$, but since we believe the description of the numerical range given in part $(a)$ may also be useful, we separate the theorem into two parts.

\begin{Theorem}\label{3by3J}
    Let \begin{equation}\label{Athmain}
        A =  \begin{pmatrix}
        0 & a_{12} & a_{13} \\
        a_{21} & 0 & a_{23} \\
        a_{31} & a_{32} & 0
    \end{pmatrix}.
    \end{equation}
\begin{enumerate}
    \item $W(A)$ is given by
\begin{equation}\label{finalJ3by3}
\left\{
  \begin{aligned}
  & s_1c_1 \,\Big(
c_2 \, J_{a_{12},a_{21}}(e^{i\psi_2}) + s_2 \, J_{a_{13},a_{31}}(e^{i\psi_3})
\Big) + s_1^2 \, s_2c_2 \, J_{a_{23},a_{32}}(e^{i(\psi_3-\psi_2)}) ,\\
  &  \text{such that }\,  \theta_1, \theta_2 \in \left[0, \frac{\pi}{2}\right], \, \psi_2, \psi_3 \in \left[0, 2 \pi\right]
  \end{aligned}
\right\},
\end{equation}
where $s_1, s_2, c_1, c_2$ are shorthands for $\sin ( \theta_1), \sin ( \theta_2), \cos( \theta_1), \cos( \theta_2)$ respectively.

\item Let $\begin{pmatrix}
    0 & a_{12} \\
    a_{21} & 0 
\end{pmatrix}$ and $\begin{pmatrix}
    0 & a_{13} \\
    a_{31} & 0 
\end{pmatrix}$ have numerical range boundaries parametrised by 
$$
\begin{aligned}
    &e^{i \alpha}(X_1\cos(t) + iY_1 \sin(t)), \,  \text{where}  \, X_1,Y_1 \geq 0 \, \text{and} \, \alpha \in [0, 2 \pi] \\
    &e^{i \beta}(X_2 \cos(s) + i Y_2 \sin(s)), \,  \text{where}  \, X_2,Y_2 \geq 0 \, \text{and} \, \beta \in [0, 2 \pi]
\end{aligned}
$$
respectively. If $a_{32} = a_{23}=0$, then $W(A)$ is an ellipse with boundary parametrised by

$$
t \mapsto \frac{e^{i \gamma}}{2}\left(\sqrt{\lambda_1} \cos (t) + i \sqrt{\lambda_2}\sin (t)\right)
$$
where
\[
\begin{aligned}
&\gamma = \frac12 \arg\Big( (X_1^2-Y_1^2)e^{i2\alpha} + (X_2^2-Y_2^2)e^{i2\beta} \Big),\\
&\lambda_{1,2} = \frac{X_1^2 + Y_1^2 + X_2^2 + Y_2^2}{2} \; \\
&\pm \frac{1}{2}\sqrt{(X_1^2-Y_1^2)^2 + (X_2^2-Y_2^2)^2 + 2(X_1^2-Y_1^2)(X_2^2-Y_2^2)\cos\big(2(\alpha-\beta)\big)}.
\end{aligned}
\]
(Here $\lambda_1$ is chosen with the $+$ sign and $\lambda_2$ with the $-$ sign.)
\end{enumerate} 

\end{Theorem}
\begin{proof}

\begin{enumerate}
    \item

For any $h = \begin{pmatrix}
    z_1 \\
    z_2 \\
    z_3
\end{pmatrix}$ of unit norm set $|z_1|=\cos (\theta_1)$. Then $|z_2|^2 + |z_3|^2 = \sin^2(\theta_1)$, so $|z_2|=\sin (\theta_1)\cos (\theta_2)$, $|z_3|=\sin (\theta_1)\sin (\theta_2)$ 
for some $\theta_2 \in[0,\tfrac{\pi}{2}]$. Write $z_j = e^{i\phi_j}|z_j|$ for phases $\phi_j\in[0,2\pi)$, 
so
\[
h=
\begin{pmatrix}
e^{i\phi_1}\cos (\theta_1) \\
e^{i\phi_2}\sin (\theta_1)\cos (\theta_2) \\
e^{i\phi_3}\sin (\theta_1)\sin (\theta_2)
\end{pmatrix}.
\]
Thus every $h$ of unit norm takes this form for some $\phi_1, \phi_2, \phi_3, \in [0, 2 \pi], \ \theta_1, \theta_2 \in [0, \frac{\pi}{2} ]$.

Denoting $$ \tilde{h} = e^{-i\phi_1}
h =  \begin{pmatrix}
\cos (\theta_1) \\
e^{i(\phi_2 - \phi_1)}\sin (\theta_1)\cos (\theta_2) \\
e^{i(\phi_3-\phi_1)}\sin (\theta_1)\sin (\theta_2)
\end{pmatrix},
$$
and observing that $\langle Ah,h \rangle  = \langle A\tilde{h},\tilde{h} \rangle$ we see
\begin{align}
&\left\{ \langle Ah,h\rangle :  \phi_1, \phi_2, \phi_3, \in [0, 2 \pi], \ \theta_1, \theta_2 \in [0, \frac{\pi}{2} ]  \right\} \\
&= \left\{ \langle Ah,h\rangle :  \phi_1=0, \phi_2, \phi_3, \in [0, 2 \pi], \ \theta_1, \theta_2 \in [0, \frac{\pi}{2} ]  \right\}
\end{align}

Furthermore a direct computation shows $\langle A\tilde{h},\tilde{h}\rangle$ is given by
\[
\begin{aligned}
& \cos (\theta_1) \,\sin (\theta_1) \Big(
\cos (\theta_2) \,(a_{12} e^{i\phi_2} + a_{21} e^{-i\phi_2}) 
+ \sin (\theta_2) \,(a_{13} e^{i\phi_3} + a_{31} e^{-i\phi_3}) 
\Big)\\
&+ \sin^2\theta_1 \,\cos (\theta_2) \,\sin (\theta_2) \,(a_{23} e^{i(\phi_3-\phi_2)} + a_{32} e^{i(\phi_2-\phi_3)})\\
&= \cos (\theta_1) \,\sin (\theta_1) \,\Big(
\cos (\theta_2) \, J_{{a_{12}},a_{21}}(e^{i\phi_2}) + \sin (\theta_2) \, J_{a_{13},a_{31}}(e^{i\phi_3})
\Big)
\\
&+ \sin^2\theta_1 \, \cos (\theta_2) \, \sin (\theta_2) \, J_{a_{23},a_{32}}(e^{i(\phi_3-\phi_2)}),
\end{aligned}
\]
from which \eqref{finalJ3by3} follows.

\item Denote
\[
\begin{aligned}
    &E_1^{\text{rot}} = \{ e^{i\alpha} (X_1  \cos (t) + i Y_1  \sin (t)) : t \in [0,2\pi] \}, \\
    &E_2^{\text{rot}} = \{ e^{i\beta} (X_2 \cos (s) + i Y_2 \sin (s)) : s \in [0,2\pi] \},
\end{aligned}
\]
and observe from part $(a)$,
\begin{equation}\label{supportnumrange}
    W(A) = \left\{ \frac{\sin( 2 \theta_1)}{2} \,\Big(
\cos(\theta_2) \, E_{1}^{rot} + \sin(\theta_2) \, E_{2}^{rot})
\Big) :  \ \theta_1, \theta_2 \in \left[0, \frac{\pi}{2}\right] \right\}
\end{equation}
where the summation above is interpreted as a Minkowski summation.
From \eqref{supportnumrange}, we see $h_{W(A)} $ is attained when $\sin (2 \theta_1) = 1$ and furthermore since the support function of the Minkowski sum of two sets is the sum of the support function on the individual sets we have
\[
h_{W(A)}(\phi) = \frac{1}{2} \left(\max_{\theta_2 \in [0, 2 \pi]} \big( \cos(\theta_2) \, h_{E_1^{\text{rot}}}(\phi) + \sin( \theta_2)\, h_{E_2^{\text{rot}}}(\phi) \big) \right).
\]
Viewing the above as the (real) inner product  $$\left\langle \begin{pmatrix}
    h_{E_1^{\text{rot}}}(\phi) \\
    h_{E_2^{\text{rot}}}(\phi)
\end{pmatrix},  \begin{pmatrix}
    \cos( \theta_2) \\
    \sin( \theta_2)
\end{pmatrix} \right\rangle
$$
and applying Cauchy-Schwarz, we obtain
\begin{equation}\label{support2}
h_{W(A)}(\phi)=\frac{1}{2} \left( \sqrt{ h_{E_1^{\text{rot}}}(\phi)^2 + h_{E_2^{\text{rot}}}(\phi)^2 } \right).
\end{equation}

By Lemma \ref{ellipsesupport} for $E_1 = e^{-i \alpha}E_1^{\text{rot}}, \, E_2 = e^{-i \beta}E_2^{\text{rot}}$
\[
h_{E_1}(\phi) = \sqrt{X_1^2 \cos^2 (\phi) + Y_1^2 \sin^2 (\phi)},\quad
h_{E_2}(\phi) = \sqrt{X_2^2 \cos^2 (\phi) + Y_2^2 \sin^2 (\phi)},
\]
so 
\[
\begin{aligned}
&h_{E_1^{\text{rot}}}(\phi) = \sqrt{X_1^2 \cos^2(\phi-\alpha) + Y_1^2 \sin^2(\phi-\alpha)}\\
&h_{E_2^{\text{rot}}}(\phi) = \sqrt{X_2^2 \cos^2(\phi-\beta) + Y_2^2 \sin^2(\phi-\beta)}.
\end{aligned}
\]

For $E_1^{\mathrm{rot}}$ and $E_2^{\mathrm{rot}}$ we write the squared support function in quadratic form.  
Let 
\[
u = \begin{pmatrix} \cos\phi \\  \sin\phi \end{pmatrix}.
\]

Then
\[
h_{E_1^{\mathrm{rot}}}(\phi)^2 
= u^T \, R(\alpha)
\begin{pmatrix} X_1^2 & 0 \\  0 & Y_1^2 \end{pmatrix}
R(\alpha)^T \, u,
\]
\[
h_{E_2^{\mathrm{rot}}}(\phi)^2 
= u^T \, R(\beta)
\begin{pmatrix} X_2^2 & 0 \\  0 & Y_2^2 \end{pmatrix}
R(\beta)^T \, u,
\]
where \[
R(\alpha) =
\begin{pmatrix}
\cos\alpha & -\sin\alpha \\
\sin\alpha & \cos\alpha
\end{pmatrix}
\]
is the unitary rotation matrix.

Hence from \eqref{support2}, $h_{W(A)}(\phi)^2 = \frac{1}{4}u^T M u,$ where
\begin{equation}\label{M}
    M = R(\alpha)\begin{pmatrix}X_1^2&0\\0&Y_1^2\end{pmatrix}R(\alpha)^T
\;+\;R(\beta)\begin{pmatrix}X_2^2&0\\0&Y_2^2\end{pmatrix}R(\beta)^T.
\end{equation}

As $M$ is the sum of two real symmetric positive definite matrices, $M$ is also real symmetric and positive definite, so it can be unitarily diagonalised as
\[
M = R(\gamma)
\begin{pmatrix}\lambda_1 & 0 \\  0 & \lambda_2\end{pmatrix}
R(\gamma)^T, \qquad \lambda_1,\lambda_2 \geq 0.
\]

Therefore
\[
h_{W(A)}(\phi) =\frac{1}{2} \sqrt{\lambda_1 \cos^2(\phi-\gamma) + \lambda_2 \sin^2(\phi-\gamma)}.
\]
This is exactly the support function of an ellipse centered at the origin, rotated by angle $\gamma$, with semi-axes $\frac{\sqrt{\lambda_1}}{2}, \frac{\sqrt{\lambda_2}}{2}$.

We now compute the values $\gamma, \lambda_1, \lambda_2$. Denote
\[
M = \begin{pmatrix} m_{11} & m_{12} \\[2mm] m_{12} & m_{22} \end{pmatrix},
\]
then directly computing from \eqref{M} gives
\[
\begin{aligned}
m_{11} &= X_1^2 \cos^2\alpha + Y_1^2 \sin^2\alpha + X_2^2 \cos^2\beta + Y_2^2 \sin^2\beta,\\[1mm]
m_{22} &= X_1^2 \sin^2\alpha + Y_1^2 \cos^2\alpha + X_2^2 \sin^2\beta + Y_2^2 \cos^2\beta,\\[1mm]
m_{12} &= (X_1^2 - Y_1^2)\sin\alpha \cos\alpha + (X_2^2 - Y_2^2)\sin\beta \cos\beta.
\end{aligned}
\]
So
\[
\begin{aligned}
m_{11}-m_{22} &= (X_1^2 - Y_1^2)\cos(2\alpha) + (X_2^2 - Y_2^2)\cos(2\beta),\\[1mm]
2 m_{12} &= (X_1^2 - Y_1^2)\sin(2\alpha) + (X_2^2 - Y_2^2)\sin(2\beta),\\[1mm]
m_{11}+m_{22} &= X_1^2 + Y_1^2 + X_2^2 + Y_2^2.
\end{aligned}
\]

As $M$ is symmetric, the eigenvalues \(\lambda_1, \lambda_2\) of \(M\) are

\begin{equation}\label{lambda12}
    \lambda_{1,2} = \frac{m_{11}+m_{22}}{2} \;\pm\; \frac{1}{2}\sqrt{(m_{11}-m_{22})^2 + 4 m_{12}^2}.
\end{equation}

Substituting the $m_{ij}$ values into the above gives
\[
\begin{aligned}
&\lambda_{1,2} = \frac{X_1^2 + Y_1^2 + X_2^2 + Y_2^2}{2} \; \\
&\pm \frac{1}{2}\sqrt{(X_1^2-Y_1^2)^2 + (X_2^2-Y_2^2)^2 + 2(X_1^2-Y_1^2)(X_2^2-Y_2^2)\cos\big(2(\alpha-\beta)\big)}.
\end{aligned}
\]

As $M$ is symmetric, for $
\theta = \tfrac{1}{2}\arctan\!\left(\frac{2m_{12}}{m_{11}-m_{22}}\right)$ (this is often called the principal axis orientation)
the eigenvectors of $M$ are
\[
v_1 = \begin{pmatrix}
\cos\theta \\ 
\sin\theta
\end{pmatrix},
\qquad
v_2 = \begin{pmatrix}
-\sin\theta \\ 
\cos\theta
\end{pmatrix},
\]
with $v_1$ corresponding to the larger eigenvalue (i.e. the eigenvalue with $+$ in \eqref{lambda12}).

Thus
\[
\gamma = \frac{1}{2} \arctan \left( \frac{(X_1^2-Y_1^2)\sin(2\alpha) + (X_2^2-Y_2^2)\sin(2\beta)}{ (X_1^2-Y_1^2)\cos(2\alpha) + (X_2^2-Y_2^2)\cos(2\beta)} \right),
\]
or equivalently
\[
\gamma = \frac12 \arg\Big( (X_1^2-Y_1^2)e^{i2\alpha} + (X_2^2-Y_2^2)e^{i2\beta} \Big).
\]

\end{enumerate}
\end{proof}

For a $2 \times 2$ matrix $B$, and $b \in \mathbb{C}$, the direct sum of $B$ and $b$,
$$
\begin{pmatrix}
    B & 0 \\
    0 & b
\end{pmatrix},
$$
has numerical range given by $\operatorname{conv}(\{W(B), b \})$. Since the numerical range described in part $(b)$ of the above is an ellipse, this may lead one to question whether the class of matrices considered are all unitarily equivalent to a direct sum of a $2 \times 2$ matrix and a $1 \times 1$ scalar matrix. However, this is not the case, since if $A$ is given by \eqref{Athmain} with $a_{23} = a_{32} = 0$ and there exists a unitary matrix, $U$, such that 
$$
U^*A U = \begin{pmatrix}
    B & 0 \\
    0 & b
\end{pmatrix},
$$
then 
$$
U^*(A^*A - AA^*)U = \begin{pmatrix}
    B^*B - BB^* & 0 \\
    0 & 0
\end{pmatrix}.
$$
Thus, this would imply $\det (A^*A-AA^*)  =0 $, but clearly the determinant of
\[
A^*A - AA^* \;=\;
\begin{pmatrix}
|a_{21}|^2 + |a_{31}|^2 - (|a_{12}|^2 + |a_{13}|^2) & 0 & 0 \\ 
0 & |a_{12}|^2 - |a_{21}|^2 & \overline{a_{12}}\,a_{13} - a_{21}\,\overline{a_{31}} \\ 
0 & \overline{a_{13}}\,a_{12} - a_{31}\,\overline{a_{21}} & |a_{13}|^2 - |a_{31}|^2
\end{pmatrix}
\]
may be non-zero.

We present two examples of Theorem~\ref{3by3J}(b) to illustrate the effectiveness of this approach compared with the method based on Kippenhahn polynomials.

\begin{Example}
    Let $A= \begin{pmatrix}
        0 & 2 & 0 \\
        0 & 0 & 0\\
        2 & 0 & 0
    \end{pmatrix}$. Then, with the notation given in Theorem \ref{3by3J}, one can check from Theorem \ref{2ndproof} that $X_1=Y_1=X_2=Y_2 = 1$ and $\alpha $ and $ \beta $ can be chosen to be 0. Thus $W(A)$ has natural parametrisation given by $z(\phi) = \frac{\sqrt{2}}{2}e^{i \phi}$, and trivially $W(A)$ is a circle radius $\frac{\sqrt{2}}{2}$.
\end{Example}

\begin{Example}
    Let $A= \begin{pmatrix}
        0 & 8 & 2i \\
        4 & 0 & 0 \\
        4i & 0 & 0
    \end{pmatrix}$. Then, with the notation in Theorem \ref{3by3J}, one obtains from Theorem \ref{2ndproof} that $X_1=6, \, Y_1=2, \, X_2=3, \, Y_2=1$ and $\alpha =0 , \, \beta = \frac{\pi}{2}$. Elementary arithmetic gives $\lambda_1 = 37, \, \lambda_2 = 13, \, \gamma = 0$, thus the boundary of $W(A)$ is parametrised by
    $$
    t \mapsto \frac{1}{2}\left(\sqrt{37}\ \cos (t) + \sqrt{13}i \sin (t)\right).
    $$
\end{Example}
\end{subsection}
\end{section}

\section{Proving the Elliptical Range Theorem via Specht's Theorem}\label{2.3}
In this section we use Specht's Theorem \ref{Specht} to produce a novel proof of the Elliptical Range Theorem.

% \begin{Lemma}\label{2.1}
%     Let $A$ be an $n \times n$ matrix. Then $p \in W(A)$ if and only if $A$ is unitarily equivalent to a matrix $X$ where $x_{1,1} =p$. Furthermore, in this case, if $p = \langle Ah,h \rangle$ for $\|h\| = 1$, then  \begin{equation}\label{new1}
%         \sum_{i=2}^n |x_{i,1}|^2 = \|Ah\|^2 - |\langle Ah,h \rangle |^2.
%     \end{equation} 
% \end{Lemma}

% \begin{proof}
% The backward implication follows since numerical ranges are invariant under unitary equivalences and clearly $p \in w(X)$.

% To prove the forward implication let $h$ be a unit norm vector such that $p= \langle Ah,h \rangle $ and let $h, h_2, \ldots , h_n$ be an orthonormal basis. Let $X$ be the matrix representation of $A$ with respect to $h, h_2, \ldots , h_n$. Then clearly $A$ and $X$ are unitarily equivalent and $x_{1,1} = p$.

% To prove the furthermore statement, let $U^*AU = X$, and then $x_{1,1} = \langle Xe_1,e_1 \rangle  = \langle Ah,h \rangle . $ Then $\|Ah\|^2 = \|UXU^*h\|^2 = $

% $\|Ah\|^2 - |\langle Ah,h \rangle |^2 = \sum_{i=2}^n |\langle Ah,h_i \rangle |^2 = \sum_{i=2}^n |x_{i,1}|^2 $.

% To prove the furthermore statement note that $\|Ah\|^2 - |\langle Ah,h \rangle |^2 = \sum_{i=2}^n |\langle Ah,h_i \rangle |^2 = \sum_{i=2}^n |x_{i,1}|^2 $.
% \end{proof}

\begin{Lemma}\label{2.1}
    For a $2 \times 2$ matrix $A$, we have $p \in W(A) $ if and only if $A$ is unitarily equivalent to a matrix of the form \begin{equation}\label{Apre}
        X= \begin{pmatrix}
    p & x_{12} \\
    x_{21} & x_{22}
    \end{pmatrix},
    \end{equation}
for some $x_{12}, x_{21}, x_{22} \in \mathbb{C}$.
\end{Lemma}

\begin{proof}
    Clearly any matrix of the form \eqref{Apre} will contain $p$ in its numerical range, and thus so will any matrix unitarily equivalent to a matrix of the form \eqref{Apre}. 

    Conversely, if $p \in W(A) $ then there exists a $y \in \mathbb{C}^2$ of unit norm such that $\langle Ay , y \rangle  = p$. In this case, with respect to the orthonormal basis $y, y_2 \in (\operatorname{span}y)^{\perp}$, $A$ is of the form \eqref{Apre}.
\end{proof}

The following theorem characterises the geometry of the numerical range of a $2 \times 2$ matrix.

\begin{Theorem}\label{main2}
    For a $2 \times 2$ matrix $$
    A = \begin{pmatrix}
        a_{11} & a_{12} \\
        a_{21} & a_{22},
    \end{pmatrix}
    $$ we have $ W(A) = \{ p \in \mathbb{C} : 2|p(\tr(A)-p) - \det(A)| \leqslant \|A\|_{HS}^2 - |p|^2 -  |\tr(A)-p|^2 \}.$
\end{Theorem}
\begin{proof}

Combining Lemma \ref{2.1} and Specht's Theorem \ref{Specht} shows that $p \in W(A)$ if and only if there exist $x_{12}, x_{21}, x_{22} \in \mathbb{C}$ such that
\begin{align}
    \tr(A):= a_{11} + a_{22} &=p+ x_{22} \label{1} \\
    a_{11}^2 + 2a_{12}a_{21} + a_{22}^2 &= p^2 + 2x_{12}x_{21} + x_{22}^2 \label{2}\\
     \|A\|_{HS}^2  = |a_{11}|^2 + |a_{12}|^2 + |a_{21}|^2 + |a_{22}|^2 &= |p|^2+ |x_{12}|^2 + |x_{21}|^2 + |x_{22}|^2. \label{3}
\end{align}
We first assume that there are values of $x$ which satisfy the three equations above, then squaring \eqref{1} and subtracting \eqref{2}, we see \eqref{1} and \eqref{2} is equivalent to \eqref{1} and \begin{equation}\label{2new}
    px_{22}-x_{12} x_{21} = \det A .
\end{equation}
%\begin{Lemma}
%There exists $x_{12}, x_{21}, x_{22} \in \mathbb{C}$ such that
%\begin{align}
%    \tr(A):= a_{11} + a_{22} &=p+ x_{22} \label{1} \\
%     px_{22}-x_{12} x_{21} = \det A\label{2new}\\
%     \|A\|_{HS}^2 : = |a_{11}|^2 + |a_{12}|^2 + |a_{21}|^2 + |a_{22}|^2 &= |p|^2+ |x_{12}|^2 + |x_{21}|^2 + |x_{22}|^2 \label{3}
%\end{align}
%if and only if there exists a $x_{12} \in \mathbb{C}$ such that 
%\begin{equation}\label{final}
%    |p|^2 + |x_{12}|^2  + \left|\frac{p(\tr(A)-p) - \det(A)}{x_{12}}\right|^2 +|\tr(A)-p|^2 \leqslant \|A\|_{HS}^2.
%\end{equation}
%\end{Lemma}
Substituting \eqref{1} and \eqref{2new} into \eqref{3} gives $|p|^2 + |x_{12}|^2  + \left|\frac{px_{22} - \det(A)}{x_{12}}\right|^2 +|\tr(A)-p|^2 = \|A\|_{HS}^2.$ Using \eqref{1} to eliminate $x_{22}$ yields
\begin{equation}\label{allequiv}
    |p|^2 + |x_{12}|^2  + \left|\frac{p(\tr(A)-p) - \det(A)}{x_{12}}\right|^2 +|\tr(A)-p|^2 = \|A\|_{HS}^2.
\end{equation}

On the other hand, starting with equation \eqref{allequiv}, if we set $x_{12}, x_{21}, x_{22}$ such that $\tr(A) - p = x_{22}$ and $x_{21} = \frac{p x_{22} - \det(A)}{x_{12}}$, then we see that \eqref{1} and \eqref{2new} are satisfied. Then substituting these values equation \eqref{allequiv} yields \eqref{3}.

Thus there exist $x_{12}, x_{21}, x_{22} \in \mathbb{C}$ such that equations \eqref{1} \eqref{2} \eqref{3} are satisfied if and only if there exists an $x_{12}$ such that equation \eqref{allequiv} holds. Clearly as $|x_{12}| \rightarrow \infty$ the left hand side of the above tends to infinity, so satisfying \eqref{allequiv} is equivalent to the existence of an $x_{12} \in \mathbb{C}$ such that 
\begin{equation}\label{ineq}
|p|^2 + |x_{12}|^2  + \left|\frac{p(\tr(A)-p) - \det(A)}{x_{12}}\right|^2 +|\tr(A)-p|^2 \leqslant \|A\|_{HS}^2.
\end{equation}
The left hand side of the above is minimised when $|x_{12}|=\sqrt{|p(\tr(A)-p) - \det(A)|}$, which means \eqref{ineq} is equivalent to $2|p(\tr(A)-p) - \det(A)| \leqslant \|A\|_{HS}^2 - |p|^2 -  |\tr(A)-p|^2.$
\end{proof}

Before we use our alternative description of the numerical range given above to recover the classical statement of the Elliptical Range Theorem, we first need a technical lemma.

\begin{Lemma}
    For a $ 2 \times 2$ matrix $A$, $\partial(W(A))$ is an ellipse parametrised by $\phi(t) = a \cos(t) + ib \sin(t)$, $a,b >0$ if and only if $ \{ z^2 : z \in \partial(W(A)) \} $ is an ellipse parametrised by $ s \mapsto \frac{a^2 - b^2}{2} + \frac{a^2 + b^2}{2} \cos(s) + i ab \sin(s)$.
\end{Lemma}
\begin{proof}
    Observe $\phi(t)^2 = a^2 \cos^2(t) - b^2 \sin^2(t) + i2ab\cos(t)\sin(t) = \frac{a^2 - b^2}{2} + \frac{a^2 + b^2}{2} \cos(2t) + i ab \sin(2t)$, where the final equality follows from trigonometric double angle formulas. 
    Thus if $\partial(W(A))$ is an ellipse parametrised by $\phi(t)$, then setting $2t=s$ we see $ \{ z^2 : z \in \partial(W(A)) \} $ is parametrised by $ s \mapsto \frac{a^2 - b^2}{2} + \frac{a^2 + b^2}{2} \cos(s) + i ab \sin(s)$.
    
    We now prove the backward implication. If $$ \{ z^2 : z \in \partial(W(A)) \}  = \left\{ \frac{a^2 - b^2}{2} + \frac{a^2 + b^2}{2} \cos(s) + i ab \sin(s) : s \in [0,2 \pi] \right\},$$ then for all $z \in \partial W(A)$, \begin{equation}\label{star}
        z= +\left(a \cos\left(\frac{s}{2}\right) + ib \sin\left(\frac{s}{2}\right)\right) \, \, \text{   or } z =  -\left(a \cos\left(\frac{s}{2}\right) + ib \sin\left(\frac{s}{2}\right)\right).
    \end{equation} Thus $\partial(W(A))$ is contained in 
    \begin{align}
        &\left\{ +\left(a \cos\left(\frac{s}{2}\right) + ib \sin\left(\frac{s}{2}\right)\right): s \in [0, 2 \pi] \right\} \bigcup \left\{ - \left(a \cos\left(\frac{s}{2}\right) + ib \sin\left(\frac{s}{2}\right)\right): s \in [0, 2 \pi]  \right\} \\
        &= \left\{ a \cos\left(t\right) + ib \sin\left(t\right) :t \in [0,2\pi] \right\}
    \end{align}
     where the final equality holds because $-  \left(a \cos\left(t\right) + ib \sin\left(t\right)\right) = a \cos\left(\pi + t\right) + ib \sin\left(\pi + t\right))$. 

    To show $\{ a \cos(t) + ib \sin(t) :t \in [0,2\pi] \} \subseteq \partial(W(A))$, first note that from \eqref{star}, for each $s \in [0,2\pi]$, either $z = +(a \cos(\frac{s}{2}) + ib \sin(\frac{s}{2})) \in \partial (W(A)) \, \, \text{   or } z =  -(a \cos(\frac{s}{2}) + ib \sin(\frac{s}{2})) \in \partial (W(A))$. Since numerical ranges (and the trace of a matrix) are invariant under unitary equivalences, and since $\tr(A)=0$, Schur's unitary diagonalization Theorem \cite[Theorem 5.4.11]{watkins2004fundamentals} means $W(A) = W(B)$, where $B= \begin{pmatrix}
        \lambda & y \\
        0 & - \lambda
    \end{pmatrix}$
    for some $\lambda, y \in \mathbb{C}$. Observe that $$p = \left\langle \begin{pmatrix}
        \lambda & y \\
        0 & - \lambda
    \end{pmatrix}\begin{pmatrix}
        v_1 \\
        v_2
    \end{pmatrix}, \begin{pmatrix}
        v_1 \\
        v_2
    \end{pmatrix} \right\rangle \in W(A),$$ if and only if
    \begin{equation}\label{-invariant}
        -p = \left\langle\begin{pmatrix}
        \lambda & y \\
        0 & - \lambda
    \end{pmatrix}\begin{pmatrix}
        -\overline{v_2} \\
        \overline{v_1}
    \end{pmatrix}, \begin{pmatrix}
        -\overline{v_2} \\
        \overline{v_1}
    \end{pmatrix} \right\rangle \in W(A) .
    \end{equation}
    So we must have \begin{align}
    &\left\{ +a \cos(\frac{s}{2}) + ib \sin(\frac{s}{2}) : s \in [0,2\pi]\right\} \bigcup \left\{ -(a \cos(\frac{s}{2}) + ib \sin(\frac{s}{2})) :  s \in [0,2\pi] \right\} \\
    &\subseteq \partial(W(A)).\end{align}
\end{proof}

\begin{Remark}
    In the above proof of the backward implication it may be tempting to use convexity to simplify the argument, but since (most) proofs of the Toeplitz-Hausdorff Theorem reduce to showing convexity of the numerical range of a $2 \times 2$ matrix, using convexity of the numerical range would be circular logic. 
\end{Remark}

We can make the following corollary to Theorem \ref{main2}, to recover a version of the Elliptical Range Theorem.

\begin{Corollary}
Let $A$ be a $ 2 \times 2$ matrix. Let $\theta = - \frac{\operatorname{arg}(\det (A - \frac{\operatorname{tr}(A)}{2})}{2} $ and $A_0 = e^{i \theta}(A - \frac{\operatorname{tr}(A)}{2})$. Then $W(A)$ is an ellipse with boundary parametrised by 
$$
t \mapsto \frac{\operatorname{tr}(A)}{2} +  e^{-i\theta}\left(\sqrt{\frac{\|A_0\|_{HS}^2}{4} - \frac{\det(A_0)}{2}} \cos(t) + i \sqrt{\frac{\|A_0\|_{HS}^2}{4} + \frac{\det(A_0)}{2}} \sin(t)\right).
$$ 
\end{Corollary}
\begin{proof}
Since $\operatorname{tr}(A_0) = 0$ and $\det A_0 \in \mathbb{R}$ Theorem \ref{main2} rearranges to $ W(A_0) = \left\{ p \in \mathbb{C} : |p^2 + \det(A_0)| + |p^2| \leqslant \frac{\|A_0\|_{HS}^2}{2} \right\}$. Thus $p^2$ is a locus of points with Foci given by $(- \det(A_0), 0)$ and major axis $\frac{\|A_0\|_{HS}^2}{2}$. Hence, by the locus of points definition of an ellipse,  $\{ p^2 : p \in W(A_0) \}$ is an ellipse with boundary parametrised by $-\frac{\det(A_0)}{2} + \frac{\|A_0\|_{HS}^2}{4} \cos(t) + i \sqrt{\frac{\|A_0\|_{HS}^4 - 4\det(A_0)^2}{16}} \sin(t)$ and by the previous lemma, $W(A_0)$ is an ellipse with boundary parametrised by 
$$
t \mapsto \sqrt{\frac{\|A_0\|_{HS}^2}{4} - \frac{\det(A_0)}{2}} \cos(t) + i \sqrt{\frac{\|A_0\|_{HS}^2}{4} + \frac{\det(A_0)}{2}} \sin(t).
$$
The general case now follows via the linearity relation $W(A_0) = e^{i \theta}W(A) - e^{i \theta}\tr(A)$.
\end{proof}

\begin{Remark}
    The method of proof for Theorem \ref{main2} generalises to the case of $3 \times 3$ matrices. Indeed, for a $3 \times 3$ matrix $A$, we have $p \in W(A)$ if and only if $A$ is unitarily equivalent to a matrix with a $p$ in the top left entry. Furthermore part $(b)$ of Theorem \ref{Specht} gives 7 trace equalities which are equivalent to $3 \times 3$ matrices being unitarily equivalent. The difficulty in generalising the approach of this section to the $ 3 \times 3$ case is handling the seven trace inequalities (analogous to \eqref{1}--\eqref{3} in the \(2\times 2\) case) in order to extract a geometric description of the numerical range.
 
\end{Remark}

\begin{section}{$C$-Numerical Range Description}\label{3n}
For an $n$-by-$n$ matrix $A$, let $\mathcal{U}(A) := \{U^*AU :  U \text{ is unitary} \}$ be the unitary orbit of $A$. Then for an $n$-by-$n$ matrix $C$, the $C$-\textit{numerical range} of $A$ is 
$$
W_C(A) = \{\tr(CX) : X \in \mathcal{U}(A) \}.
$$

The $C$-numerical range generalises the numerical range in the sense that if $C$ is a matrix with a $1$ in the top left entry and 0 elsewhere, we recover the classical numerical range. $C$-numerical ranges have been extensively studied, and we refer the reader to \cite{li1994c} for a survey on the topic.

\begin{subsection}{$2 \times 2$ $C$-Numerical Range Description}\label{3.1}

The $C$-numerical range of a $2 \times 2$ matrix was first shown to be an ellipse in \cite{nakazato1994c}, and a subsequent proof was given in \cite{kwong1996some}. The following proof partially uses the argument made in \cite{kwong1996some}, but we are able to extend the result in \cite{kwong1996some} by explicitly identifying the ellipse that represents the $C$-numerical range of an arbitrary matrix.
\begin{Theorem}
    Let $A$ and $C$ be $2 \times 2$ matrices. Let $\theta_1 = \frac{\pi - \operatorname{arg}(\det (A - \frac{\operatorname{tr}(A)}{2}))}{2} $, $A_0 = e^{i \theta_1}(A - \frac{\operatorname{tr}(A)}{2})$, $\theta_2 = \frac{\pi - \operatorname{arg}(\det (C - \frac{\operatorname{tr}(C)}{2}))}{2} $ and $C_0 = e^{i \theta_2}(C - \frac{\operatorname{tr}(C)}{2})$. Then $W_C(A)$ is the ellipse with boundary parametrised by
    \begin{equation}\label{1stC}
    t \mapsto  \frac{\operatorname{tr}(A) \operatorname{tr}(C)}{2} + e^{-i(\theta_1 +  \theta_2)}(K_1  \cos (t) + i K_2 \sin (t)),
    \end{equation}
    where $K_1^2$ is given by
$$
        \|A_0\|_{HS}^2\,\|C_0\|_{HS}^2
+\sqrt{\|A_0\|_{HS}^4-4(\det A_0)^2}\;
 \sqrt{\|C_0\|_{HS}^4-4(\det C_0)^2}
+4\,\det(A_0)\,\det(C_0),
$$
and $K_2^2$ is given by
$$ \|A_0\|_{HS}^2\,\|C_0\|_{HS}^2
+\sqrt{\|A_0\|_{HS}^4-4(\det A_0)^2}\;
 \sqrt{\|C_0\|_{HS}^4-4(\det C_0)^2}
-4\,\det(A_0)\,\det(C_0).
$$
\end{Theorem}
\begin{proof}
As noted in the proof of Theorem \ref{2ndproof}, $A_0$ and $C_0$ are unitarily equivalent to $s_1\begin{pmatrix}
    0 & b_1 \\
    c_1 & 0
\end{pmatrix}$ and $s_2\begin{pmatrix}
    0 & b_2 \\
    c_2 & 0
\end{pmatrix}$ respectively, where $s_i = \pm 1$ and $b_i \geq c_i \geq 0$. We first show $W_{C_0}(A_0) = \operatorname{conv}(\ran J_{c_1c_2, b_1b_2})$. We first consider when $s_1 = s_2 =1$. Since every $ 2 \times 2$ unitary matrix is of the form $\lambda \begin{pmatrix}
    \cos (\theta) & e^{i \phi_1} \sin (\theta) \\
    e^{i \phi_2} \sin (\theta) & -e^{i (\phi_1+ \phi_2)} \cos (\theta)
\end{pmatrix}$ for $|\lambda| =  1$, and $\theta, \phi_1, \phi_2 \in [0, 2 \pi]$ and $C$-numerical ranges are invariant under unitary equivalences (Proposition 2.1 in \cite{gau2021numerical}), directly computing the trace of
$$
\begin{pmatrix}
    0 & b_2 \\
    c_2 & 0
\end{pmatrix} \begin{pmatrix}
    \cos (\theta) & e^{-i \phi_2} \sin (\theta) \\
     e^{-i \phi_1} \sin (\theta)& -e^{-i (\phi_1+ \phi_2)} \cos (\theta)
\end{pmatrix}   \begin{pmatrix}
    0 & b_1 \\
    c_1 & 0
\end{pmatrix}\begin{pmatrix}
    \cos (\theta) & e^{i \phi_1} \sin (\theta) \\
    e^{i \phi_2} \sin (\theta) & -e^{i (\phi_1+ \phi_2)} \cos (\theta)
\end{pmatrix}
$$
gives $W_{C_0}(A_0) $ as
\begin{equation}\label{Cmain}
\{ \sin^2 (\theta) J_{c_1c_2,b_1b_2}(e^{i( \phi_1 - \phi_2)}  ) +  \cos^2( \theta) J_{c_2b_1, b_2c_1} (-e^{i (\phi_1 +\phi_2) }):\theta, \phi_1, \phi_2 \in [0, 2 \pi] \}.
\end{equation}
Thus every point in $W_{C_0}(A_0)$ is a convex combination of a point in $ \ran J_{c_1c_2, b_1b_2}$ and $ \ran J_{c_2b_1, b_2c_1}$. Since $b_i \geq c_i$ we have \begin{align}
&c_1 c_2 + b_1 b_2 - (c_2 b_1 + b_2 c_1)
= b_1 b_2 - b_1 c_2 - b_2 c_1 + c_1 c_2
= (b_1 - c_1)(b_2 - c_2) \geq 0 \\
&(b_1 b_2 - c_1 c_2)^2 - (b_1 c_2 - b_2 c_1)^2
= (b_1 - c_1)(b_2 - c_2)(b_1 + c_1)(b_2 + c_2)
\geq 0,
\end{align}
thus from \eqref{Jtransform} we see the major and minor axis length of the ellipse $\operatorname{conv}(\ran J_{c_1c_2, b_1b_2})$ are greater than $\operatorname{conv}(\ran J_{c_2b_1, b_2c_1}) $ and thus $\ran J_{c_2b_1, b_2c_1}\subseteq \operatorname{conv}(\ran J_{c_1c_2, b_1b_2})$, and consequently $W_{C_0}(A_0) \subseteq \operatorname{conv}(\ran J_{c_1c_2, b_1b_2})$. Conversely setting $\phi_1 =i \frac{\zeta}{2}, \,  \phi_2 = -i \frac{\zeta}{2}$ in \eqref{Cmain} gives
$$
\{\sin^2(\theta) J_{c_1c_2, b_1b_2}(e^{i\zeta}) + \cos^2 (\theta) (-b_2c_1 - c_2b_1): \zeta, \theta, \in [0, 2 \pi]\}
$$
As $-b_2c_1 - c_2b_1 = J_{c_2b_1, b_2c_1}(-1) \in \operatorname{conv}(\ran J_{c_1c_2, b_1b_2})$, the above is 
$$
\operatorname{conv}(\ran J_{c_1c_2, b_1b_2}, -b_2c_1 - c_2b_1) = \operatorname{conv}(\ran J_{c_1c_2, b_1b_2}).
$$
Thus $W_{C_0}(A_0) = \operatorname{conv}(\ran J_{c_1c_2, b_1b_2})$. As $W_{C_0}(A_0)$ is symmetric in the imaginary axis the same formula will hold if we initially chose $s_1=-1$, or $s_2= -1.$

We now show that $c_1c_2 + b_1b_2 = K_1$ and $|c_1c_2 - b_1b_2| = K_2$.
Denote
\begin{align}\label{st}
&s_1^2=\frac{\|A_0\|_{HS}^2}{4}-\frac{\det A_0}{2}, \, \, \, 
t_1^2=\frac{\|A_0\|_{HS}^2}{4}+\frac{\det A_0}{2} \\
&s_2^2=\frac{\|C_0\|_{HS}^2}{4}-\frac{\det C_0}{2}, \, \, \, 
t_2^2=\frac{\|C_0\|_{HS}^2}{4}+\frac{\det C_0}{2},
\end{align}
Then \eqref{useforC} says \(b_1=s_1+t_1\), \(c_1=s_1-t_1\), \(b_2=s_2+t_2\), \(c_2=s_2-t_2\). So $b_1b_2+c_1c_2 =2(s_1s_2+t_1t_2)$ and $b_1b_2-c_1c_2 =2(s_1t_2+t_1s_2).$ Thus,
\begin{align}
(b_1b_2+c_1c_2)^2+(b_1b_2-c_1c_2)^2
&=4\!\left[(s_1s_2+t_1t_2)^2+(s_1t_2+t_1s_2)^2\right]\\
&=4(s_1^2+t_1^2)(s_2^2+t_2^2)+16\,s_1t_1s_2t_2\\
&=\|A_0\|_{HS}^2\,\|C_0\|_{HS}^2
\;+\;16\,s_1t_1s_2t_2. \label{Kadd1}
\end{align}
Similarly \begin{equation}\label{Kadd2}
    (b_1b_2+c_1c_2)^2-(b_1b_2-c_1c_2)^2 =4\,\det(A_0)\,\det(C_0).
\end{equation}
Adding \eqref{Kadd1} and \eqref{Kadd2} gives the expression for $K_1$ and subtracting \eqref{Kadd1} and \eqref{Kadd2} gives the expression for $K_2$.

The relations $W_C(A)= W_A(C)$ and $W_C(\alpha A + \beta I) =\alpha W_C(A) + \beta \operatorname{tr}(C)$ for all $\alpha , \beta \in \mathbb{C}$ (which can be found as Proposition 2.1 in \cite{gau2021numerical}) imply that $W_C(A) = e^{-i ( \theta_1 + \theta_2)}W_{C_0}(A_0) + \frac{\operatorname{tr}(A) \operatorname{tr}(C)}{2}$, and so $W_C(A)$ is given by \eqref{1stC}.
\end{proof}
\end{subsection}

\begin{subsection}{Rank 1 $C$-Numerical Ranges}\label{3.2}

The influential union of disks formula for the $q$-numerical range was first shown as Lemma 5 in \cite{tsing1984constrained}. Since $q$-numerical ranges are rank 1 $C$-numerical ranges, Theorem \ref{uniondisk} below may be viewed as a multidimensional generalisation of the union of disks formula. We remark that the proof of Theorem \ref{uniondisk} uses a different approach to that in \cite{tsing1984constrained}.

It can be shown that every rank one matrix $C$ is unitarily equivalent to a matrix with all rows equal to 0 apart form the first row. As previously mentioned, Proposition 2.1 in \cite{gau2021numerical} shows that if $C_1$ and $C_2$ are unitarily equivalent then $W_{C_1}(A) = W_{C_2}(A)$. Since the previous theorem characterises the $C$-numerical ranges of $2 \times 2$ matrices we restrict ourself to the case of $n \times n$ matrices where $n \geq 3$. Thus to characterise $C$-numerical ranges of matrices with $\operatorname{rank}C = 1$ one may assume the hypothesis in the following theorem.

    \begin{Theorem}\label{uniondisk}
        Let $A$ and $C$ be $n$-by-$n$ matrices where $n\geq 3$ and let the $i,j$-th entry of $C$ be denoted $c_{ij}$. If $c_{ij} = 0 $ for $i \geq 2$, then $W_C(A)$ is given by
        \begin{equation}\label{rank1Cfinal}
            \left\{ c_{11} \langle Ah,h \rangle + y : \|h\| = 1 , \text{ and }  |y|^2 \leq (\|Ah\|^2- |\langle Ah,h\rangle|^2) \left(\sum_{j=2}^n|c_{1j}|^2 \right) \right\}.
        \end{equation}
    \end{Theorem}

\begin{proof}
We first prove that $W_C(A)$ is given by \begin{equation}\label{rank1C}
         \left\{ c_{11} \langle Ah,h \rangle + \left\langle \begin{pmatrix}
            c_{12} \\
            c_{13} \\
            \vdots \\
            c_{1n}
        \end{pmatrix}, \begin{pmatrix}
            x_{2} \\
            x_{3} \\
            \vdots \\
            x_{n}
        \end{pmatrix} \right\rangle : \|h\| = 1 , \text{ and }  \sum_{i=2}^n|x_i|^2 = \|Ah\|^2- |\langle Ah,h\rangle|^2 \right\}
        \end{equation}

Observe that $\mathcal{U}(A) = \left\{(\langle Ah_i,h_j \rangle ) : h_1, h_2, \ldots h_n \text{ is an o.n.b for } \mathbb{C}^n  \right\}$. To show the $\subseteq $ inclusion, note that $p \in W_C(A)$ is given by $$
p = c_{11} \langle Ah,h\rangle +  \sum_{j=2}^n c_{1j}\langle Ah,h_j\rangle , \, \, \, \, h,h_2,\ldots , h_n \, \text{is an o.n.b for } \mathbb{C}^n 
$$
which is of the form given in \eqref{rank1C}, since $\sum_{j=2}^n |\langle Ah,h_j\rangle|^2 = \|Ah\|^2 - |\langle Ah,h\rangle |^2$.

To show the $\supseteq$ inclusion it suffices to show that given any $h \in \mathbb{C}^n$ of norm 1 and $x=\begin{pmatrix}
        x_2 \\
        x_3 \\
        \vdots \\
        x_n
    \end{pmatrix} \in \mathbb{C}^{n-1}$ such that $\|x\|^2 = \|Ah\|^2 - |\langle Ah,h\rangle |^2$, there exists an orthonormal basis $h_1, h_2 \ldots h_n$ such that $\sum_{j=1}^n c_{1j}\langle Ah_1,h_j \rangle  = c_{11} \langle Ah,h \rangle + \left\langle \begin{pmatrix}
            c_{12} \\
            c_{13} \\
            \vdots \\
            c_{1n}
        \end{pmatrix}, \begin{pmatrix}
            x_{2} \\
            x_{3} \\
            \vdots \\
            x_{n}
        \end{pmatrix} \right\rangle$. This can be achieved by setting $h_1 = h$ and constructing $h_2, h_3, \ldots h_n$ so that $\langle Ah_1,h_j\rangle = \overline{x_j}$ as follows. Pick a arbitrary orthonormal basis $\tilde{h_j} \, \,  \,  2 \leq j \leq n $ of $h_1^{\perp}$. Set
\[
z = Ah_1 - \langle Ah_1, h_1 \rangle h_1 \in h_1^{\perp}, 
\qquad 
b := \sum_{i=2}^n \overline{x_i} \tilde{h_i} \in h_1^{\perp},
\]
so $\|b\| = \|x\| = \|z\|$.

Set $g_1 := \frac{z}{\|z\|}, f_1 := \frac{b}{\|b\|} $ and extend $\{g_1\}$ and $\{f_1\}$ to orthonormal bases of $h_1^{\perp}$ respectively.
Define the unitary map $W:h_1^{\perp} \to h_1^{\perp}$ by $W g_j = f_j$ for all $j$, and then
\[
W z = W(\|z\| g_1) = \|z\| f_1 = b.
\]
Set $h_j = W^* \tilde{h_j}$ for $j\geq 2$, so $\langle Ah_1, h_j\rangle = \langle Wz, \tilde{h_j} \rangle  = \overline{x_j}$. Thus we have proved \eqref{rank1C}.

Equation \eqref{rank1Cfinal} follows from \eqref{rank1C} since for any $ y \in \mathbb{C}$ with $$|y|^2 \leq (\|Ah\|^2- |\langle Ah,h\rangle|^2) \left( \sum_{j=2}^n|c_{ij}|^2 \right)$$ is of the form 
$$
y = \left\langle \begin{pmatrix}
            c_{12} \\
            c_{13} \\
            \vdots \\
            c_{1n}
        \end{pmatrix}, \begin{pmatrix}
            x_{2} \\
            x_{3} \\
            \vdots \\
            x_{n}
        \end{pmatrix} \right\rangle  
$$where $$
\begin{pmatrix}
    x_2 \\
    x_3 \\
    \vdots \\
    x_n
\end{pmatrix} =  \left( r_1 \begin{pmatrix}
            c_{12} \\
            c_{13} \\
            \vdots \\
            c_{1n}
        \end{pmatrix}+ r_2 \begin{pmatrix}
            d_{12} \\
            d_{13} \\
            \vdots \\
            d_{1n}
        \end{pmatrix} \right), \quad \sum_{i=2}^n|x_i|^2 = \|Ah\|^2- |\langle Ah,h\rangle|^2,
$$
$\begin{pmatrix}
            d_{12} \\
            d_{13} \\
            \vdots \\
            d_{1n}
        \end{pmatrix} \in \begin{pmatrix}
            c_{12} \\
            c_{13} \\
            \vdots \\
            c_{1n}
        \end{pmatrix}^{\perp}$ and $ r_1, r_2 \in \mathbb{C}$ are appropriately chosen.
\end{proof}
One can deduce the following corollary from the previous theorem.
\begin{Corollary}
    Let $C_1$ and $C_2$ be $n$-by-$n$ matrices for $n\geq 3$ where $C_1 = (c_{ij}^1)$,  $C_2 = (c_{ij}^2)$, and $c_{ij}^1 = 0, \, c_{ij}^2 = 0 $ for $i \geq 2$. If $c_{11}^1 = c_{11}^2$ and $\sum_{j=2}^n|c_{1j}^1|^2  = \sum_{j=2}^n|c_{1j}^2|^2 $, then $W_{C_1}(A) = W_{C_2}(A)$. 
\end{Corollary}

\end{subsection}

    % $$
    % \mathcal{U}(A) = \left\{ \begin{pmatrix}
    %     \langle Ah_1,h_1\rangle , \langle Ah_1, h_2
    % \end{pmatrix}\right\}
\end{section}

\begin{section}{Acknowledgements}
     The author is grateful to EPSRC for financial support (grant - EP/Y008375/1).
      \newline
      The author gratefully acknowledges the ILAS 2025 organising committee, whose efforts provided the opportunity for valuable discussions with participants at ILAS 2025 that contributed to the development of this paper.

    \vskip 0.3cm

    \noindent The author has no declarations of interest.
\end{section}

 \noindent    Author email address - r.d.oloughlin@reading.ac.uk

 \newpage
 \bibliographystyle{plain}
\bibliography{bibliography.bib}
\end{document}